\documentclass{amsart}

\usepackage{amscd,amsmath,amssymb,amsfonts,bbm, calligra}
\usepackage[T1]{fontenc}
\usepackage{lmodern}
\usepackage{mathtools}
\usepackage{enumerate}
\usepackage{multicol}
\usepackage[all]{xy}
\usepackage{mathrsfs}
\usepackage{footmisc}
\usepackage[unicode,pdfborder={0 0 0},final]{hyperref}

\newtheorem{thm}{Theorem}[section]
\newtheorem{prop}[thm]{Proposition}

\newtheorem{lem}[thm]{Lemma}

\theoremstyle{definition}

\theoremstyle{remark}
\newtheorem{rem}[thm]{Remark}

\numberwithin{equation}{section}

\newcommand{\Gys}{\mathrm{Gys}}
\newcommand{\cl}{\mathrm{cl}}
\newcommand{\cd}{\mathrm{cd}}
\newcommand{\res}{\mathrm{res}}
\newcommand{\et}{\mathrm{\acute{e}t}}
\newcommand{\red}{\mathrm{red}}
\newcommand{\ret}{\mathrm{r\acute{e}t}}
\newcommand{\Spec}{\mathrm{Spec}}
\newcommand{\Frac}{\mathrm{Frac}}
\newcommand{\RR}{\mathrm{R}}

\newcommand{\cupp}{\mkern-1mu\smile\mkern-1mu}

\makeatletter
\def\myrightarrow{{\setbox\z@\hbox{$\rightarrow$}\dimen0\ht\z@\multiply\dimen0 6\divide\dimen0 10\ht\z@\dimen0\box\z@}}
\def\myrightarrowfill@{\arrowfill@\relbar\relbar\myrightarrow}
\newcommand{\myxrightarrow}[2][]{\ext@arrow 0359\myrightarrowfill@{#1}{#2}}
\makeatother

\def\bG{\mathbf G}
\def\bmu{\boldsymbol\mu}

\def\cL{\mathcal{L}}
\def\cO{\mathcal{O}}

\def\Z{\mathbb Z}

\def\F{\mathbb F}

\def\Q{\mathbb Q}

\def\R{\mathbb R}
\def\bN{\mathbb N}

\begin{document}

\title[Sums of squares in function fields over Henselian local fields]
{Sums of squares in function fields over Henselian local fields}

\author{Olivier Benoist}
\address{D\'epartement de math\'ematiques et applications, \'Ecole normale sup\'erieure, CNRS,
45 rue d'Ulm, 75230 Paris Cedex 05, France}
\email{olivier.benoist@ens.fr}

\renewcommand{\abstractname}{Abstract}
\begin{abstract}
We give upper bounds for the level and the Pythagoras number of function fields over fraction fields of integral Henselian excellent local rings. In particular, we show that the Pythagoras number of $\R((x_1,\dots,x_n))$ is $\leq 2^{n-1}$, which answers positively a question of Choi, Dai, Lam and Reznick.
\end{abstract}
\maketitle

\section*{Introduction}
\label{intro}

  In \cite[Satz 4]{Artin}, Artin proved that a real rational function $f\in\R(x_1,\dots,x_n)$ which does not take negative values is a sum of squares in $\R(x_1,\dots,x_n)$, thus solving Hilbert's 17th problem.
 It is natural to wonder about the number of squares required to write $f$ as a sum of squares. To study this question, one introduces the \textit{Pythagoras number} $p(K)\in\bN\cup\{+\infty\}$ of a field $K$: it is the smallest integer $p\in\bN$ such that all sums of squares in $K$ are sums of $p$ squares if such an integer exists, and $+\infty$ otherwise. Pfister \cite[Theorem 1]{Pfistersumme} was able to show that $p(\R(x_1,\dots,x_n))\leq 2^n$; as a consequence, a real rational function $f\in\R(x_1,\dots,x_n)$ that does not take negative values is a sum of $2^n$ squares in $\R(x_1,\dots,x_n)$.

A related invariant is the \textit{level} $s(K)\in\bN\cup\{+\infty\}$ of a field $K$: the smallest integer $s\in\bN$ such that $-1$ is a sum of $s$ squares in $K$, if such an integer exists, and $+\infty$ otherwise. By Artin and Schreier \cite[Satz 7b]{AS}, the level $s(K)$ is infinite if and only if $K$ admits a field ordering ($K$ is then said to be \textit{formally real}). Pfister has shown that if $s(K)$ is finite, then it is a power of $2$ \cite[Satz 4]{Pfisterlevel}, and that if  $K$ is moreover a field of transcendence degree $n$ over $\R$, then $s(K)\leq 2^n$ \cite[Theorem~2]{Pfistersumme}.

We refer to \cite[Chapters VIII and XI]{Lam} and \cite{Pfisterbook} for nice accounts of these results.

\vspace{1em}

As a particular case of our main statement (Theorem \ref{thmain} below), we obtain local analogues of Pfister's aforementioned theorems \cite[Theorems 1 and 2]{Pfistersumme}.

\begin{thm}
\label{consmain}
Fix $n\geq 1$ and let $K:=\R((x_1,\dots,x_n))$.

\begin{enumerate}[(i)]
\item One has $p(K)\leq 2^{n-1}$. 

\item If a finite extension $F$ of $K$ is not formally real, then ${s(F)\leq 2^{n-1}}$.
\end{enumerate}
\end{thm}

Theorem \ref{consmain} (i) was conjectured by Choi, Dai, Lam and Reznick \cite[\S 9, Problem~6 and below]{CDLR}. It was proven by them when $n\leq 2$ \cite[Corollary 5.14]{CDLR} and by Hu when $n=3$ \cite[Theorem 1.2]{Hupyth}.
In addition, Theorem \ref{consmain} (ii) had already been proven by Hu for $n=2$ \cite[Theorem 5.1]{Hupyth}.

Pfister's inequalities $p(\R(x_1,\dots,x_n))\leq 2^n$ are not known to be optimal (see \cite[\S 4~Problem 1]{PfisterICM}). The best result to date is the theorem of Cassels, Ellison and Pfister~\cite{CEP} according to which $p(\R(x_1,x_2))=4$. 
We do not know if the bounds stated in Theorem~\ref{consmain} are optimal either. They are however the best possible under the assumption that Pfister's bounds are optimal (see \cite[Corollary~2.3]{Hupyth} and Proposition~\ref{propopt}). This line of thought had already been exploited by Hu \cite[Theo\-rem~1.2]{Hupyth} to show the equality ${p(\R((x_1,x_2,x_3)))=4}$.

\vspace{1em}

In Theorem \ref{thmain}, we consider more generally function fields $F$ over the fraction field of an integral Henselian excellent local ring $A$ of dimension $\geq 1$. In this setting, our bounds depend on the dimension of $A$, on the transcendence degree of $F$ over $\Frac(A)$, as well as on the \textit{virtual cohomological $2$-dimension} $\cd_2(k[\sqrt{-1}])$ of the residue field $k$ of $A$, which is defined as the cohomological $2$-dimension
of the absolute Galois group of the field $k[\sqrt{-1}]$ in the sense of \cite[I \S 3.1]{CohoGalois}.

\begin{thm}
\label{thmain}
Let $A$ be an integral Henselian excellent local ring of dimension ${n\geq 1}$ whose residue field $k$ has characteristic $0$ and 
satisfies $\cd_2(k[\sqrt{-1}])\leq \delta$. Let $F$ be a field of transcendence degree $m$ over $K:=\Frac(A)$.
\begin{enumerate}[(i)]
\item If $F$ is not formally real, then $s(F)\leq 2^{n+m+\delta-1}$ and $p(F)\leq 2^{n+m+\delta-1}+1$.
\item If $F$ is formally real, $p(F)\leq 2^{n+m+\delta}-1$.
\item If $A$ is regular and $k$ is formally real, then $p(K)\leq 2^{n+\delta-1}$.
\end{enumerate}
\end{thm}

Theorem \ref{consmain} follows from Theorem \ref{thmain} by taking $A=\R[[x_1,\dots,x_n]]$.

The assumption that $k$ has characteristic $0$ in Theorem \ref{thmain} is not a significant restriction, as there are trivial upper bounds for $s(F)$ and $p(F)$ otherwise.
 (If $k$ has characteristic $p\geq 3$, then $s(k)\leq s(\F_p)\leq 2$, so that $s(F)\leq s(\Frac(A))\leq 2$ by henselianity, and $p(F)\leq 3$ by \cite[XI, Theorem 5.6~(2)]{Lam}. A similar argument shows that $s(F)\leq 4$ and $p(F)\leq 5$ if $k$ has characteristic $2$.)


\vspace{1em}

The Pythagoras numbers $p(F)$ of function fields $F$ over Henselian local fields as above had previously been studied in the literature for low values of $n$ and $m$. We refer to Becher, Grimm and Van Geel \cite[\S 6]{BGvG} for an analysis of the $n=m=1$ case, and to Hu's articles \cite{Hupyth, HuHasse} for various results when $n+m\leq 3$.

 A striking feature of these works is that the hypotheses made on the residue field~$k$ of $A$ are much weaker than ours: the authors only need to control sums of squares in function fields over $k$ (see for example \cite[Theorem 6.8]{BGvG}, \cite[Theo\-rem~1.1]{Hupyth} or \cite[Theorem 1.4]{HuHasse}), and not the whole cohomological $2$-dimension of~$k[\sqrt{-1}]$. We believe that our stronger hypothesis is key in obtaining higher-dimensional results.

  Let us illustrate this difference with the example of the field $F:=\Q((x_1,\dots,x_n))$ for $n\geq 3$. What was known before, as an application of Pfister's work and of the Milnor conjectures, is the inequality $p(F)\leq 2^{n+2}$ (see \cite[beginning of \S 5]{Hupyth}). Theorem~\ref{thmain}~(iii) improves on this result by showing that $p(F)\leq 2^{n+1}$. On the other hand, combining \cite[Conjecture~5.4]{Hupyth} and Jannsen's theorem \cite[Corollary~0.7]{Jannsen} yields the optimistic conjecture that $p(F)\leq 2^n$, which is only known for $n=3$ (see \cite[Corollary 4.7 (ii)]{HuHasse}).

\vspace{1em}

We prove Theorem \ref{thmain} (i) in \S\ref{level}. Our main tool is a variant of the Lefschetz-type vanishing theorem of Saito and Sato \cite[Theorem 3.2 (1)]{SS}, and the relevant material is gathered in \S\ref{preliminaries}.  Assertions (ii) and (iii) of Theorem \ref{thmain} are consequences of Theorem \ref{thmain} (i), as we show in \S\ref{pythagoras}. 
The former is easy, and the latter relies prominently on Panin's proof of the Gersten conjecture for regular schemes of charac\-teristic $0$ \cite[Theorem C]{Panin}. 
The optimality of Theorem \ref{consmain} is discussed~in~\S\ref{optimality}.

\vspace{1em}

\subsection*{Acknowledgements}

Our use of the Saito--Sato vanishing theorem was inspired by the article \cite{KS} of Kerz and Saito (see \emph{loc.\ cit.}, (3.11)). We are also grateful to Yong Hu for useful comments.

\subsection*{Notation and conventions}
A \textit{variety} over a field $k$ is a separated scheme of finite type over $k$. We use $k[\sqrt{-1}]$ as a notation for $k$ if $-1$ is a square in $k$ and for $k[T]/(T^2+1)$ otherwise.
We let $\cd_2(X)$ be the cohomological $2$-dimension of the \'etale site of a scheme $X$ (see \cite[Definition 7.1]{Scheiderer}). If $k$ is a field, we use the notation $\cd_2(k):=\cd_2(\Spec(k))$.

If $X$ is a scheme and $x\in X$ is a point, we let $\kappa(x)$ be the residue field of $X$ at~$x$.
The \textit{real spectrum} $X_r$ of a scheme $X$ is the set of pairs $(x,\prec)$, where $x\in X$ and $\prec$ is a field ordering of $\kappa(x)$, endowed with its natural topology \cite[(0.4)]{Scheiderer}. 

A reduced Cartier divisor~$D$ in a regular scheme $X$ is said to have \textit{simple normal crossings} if for all $c\geq 1$ and any collection $D_1,\dots,D_c$ of distinct irreducible components of $D$, the scheme-theoretic intersection $D_1\cap\dots\cap D_c$ is either empty or regular of codimension $c$ in $X$. 

If $S$ is a local scheme with closed point $s\in S$, 
and $\pi:X\to S$ is a morphism, we denote by $X_s:=\pi^{-1}(s)$ the special fiber of $\pi$.
If $S$ is quasi-excellent and $\kappa(s)$ has characteristic $0$, then separated schemes of finite type over $S$ and coherent ideal sheaves on them admit resolutions of singularities (Hironaka's theorems \cite{Hironaka} apply as indicated p.151 of \emph{loc.\ cit.}, see also \cite[Theorem 1.1.11]{Temkin}).


\section{Preliminaries}
\label{preliminaries}

We gather here two results that will be used in the proof of Theorem \ref{thmain} (i).

\subsection{A purity result of Saito and Sato}

If $i:D\to X$ is the inclusion of a Cartier divisor in a Noetherian scheme $X$, and if $N\geq 1$ is invertible on $X$, we let $\cl_{X,N}(D)\in H^2_{\et,D}(X,\bmu_N)$ be the \textit{cycle class} of $D$ in~$X$ \cite[\S 2.1]{Cycle}. In view of the canonical isomorphism $H^2_{\et,D}(X,\bmu_N)=H^2_{\et}(D,\RR i^!\bmu_N)$, it gives rise to a morphism $\Gys_{i,N}:\Z/N\to \RR i^!\bmu_N[2]$ in $D^+_{\et}(D)$ called the \textit{Gysin morphism}. Gabber's absolute purity theorem (see \cite[Theorem~2.1.1]{Fujiwara} or \cite[Th\'eor\`eme 3.1.1]{Rioupurity}) implies that $\Gys_{i,N}$ is an isomorphism if $X$ and $D$ are regular. Building on Gabber's theorem, and extending earlier results of Rapoport and Zink \cite[Lemma 2.18, Satz~2.19]{RZ}, Saito and Sato \cite[Lemma 3.4]{SS} have proven (a variant of) the following statement.

\begin{prop}
\label{purity}
Let $X$ be a regular Noetherian scheme, and let $N\geq 1$ be invertible on $X$. Let $D$ and $E$ be two Cartier divisors on $X$ that have no irreducible component in common, such that $D$ is regular, and such that $D\cup E$ is a simple normal crossings divisor on~$X$. We let $i:D\to X$, $j:X\setminus D\to X$, $i':D\cap E\to E$ and $j':E\setminus (D\cap E)\to E$ be the natural inclusions.

\begin{enumerate}[(i)]
\item
The Gysin morphism $\Gys_{i',N}$ is an isomorphism.
\item 
The restriction morphism $(\RR j_*\Z/N)|_E\to \RR j'_{*}\Z/N$ is an isomorphism.
\item Assume moreover that $X$ is proper over a local Henselian Noetherian scheme, and that $E$ is the reduced special fiber of $X$. Then, for all $q,l\in \Z$, the restriction maps $H^q_{\et}(X\setminus D,\bmu_N^{\otimes l})\to H^q_{\et}(E\setminus (D\cap E),\bmu_N^{\otimes l})$ are isomorphisms.
\end{enumerate}
\end{prop}

\begin{proof}
Assertion (i) is exactly what is shown in the proof of \cite[Lemma 3.4 (1)]{SS}. In \emph{loc.\ cit.}, the additional assumptions that $X$ is flat of finite type over a discrete va\-luation ring and that $E$ is the reduced special fiber of $X$ are not used, and $D$ and $E$ are respectively denoted by $Y$ and $Z$.

To prove (ii) and (iii), we argue as in the proof of \cite[Lemma 3.4 (2)]{SS}. In the following natural morphism of distinguished triangles in $D^+_{\et}(E)$:
\begin{equation*}
\begin{aligned}
\xymatrix
@R=0.3cm 
@C=0.45cm 
{
(i_{*}\RR i^!\Z/N)|_E\ar^{\wr}[d]\ar[r]&\Z/N\ar@{=}[d]\ar[r]& ( \RR j_{*}\Z/N)|_E\ar[d]\ar[r]&\\
i'_{*}\RR i'^!\Z/N\ar[r]&\Z/N\ar[r]&\RR j'_{*}\Z/N\ar[r]&
}
\end{aligned}
\end{equation*}
the left vertical arrow is an isomorphism since $\Gys_{i,N}$ and $\Gys_{i',N}$ are isomorphisms by Gabber's purity theorem and by (i), and since $\cl_{X,N}(D)|_E=\cl_{E,N}(D\cap E)$ by functoriality of the cycle class \cite[\S 2.1.1]{Cycle}. Assertion (ii) follows. To deduce~(iii) from (ii), tensor with $\bmu_N^{\otimes l}$, take cohomology, and apply the proper base change theorem \cite[Expos\'e XII, Corollaire 5.5 (iii)]{SGA43} and the invariance of \'etale cohomology under nilpotent closed immersions \cite[Expos\'e VIII, Corollaire~1.2]{SGA42}.
\end{proof}

\begin{rem}
  In \S\ref{level}, we will apply Proposition \ref{purity} to a scheme $X$ of characteristic~$0$. In this case, one can replace the use of Gabber's purity theorem in the proof of Proposition \ref{purity} by the earlier \cite[Expos\'e XIX, Th\'eor\`emes 3.2 et 3.4]{SGA43}.
\end{rem}

\subsection{A Bertini theorem over a local base}

Proposition \ref{Bertini} is an analogue of Jannsen and Saito's Bertini theorem over a discrete valuation ring  \cite[Theorem~1.1]{JS}, when the base has higher dimension.

\begin{prop}
\label{Bertini}
Let $S$ be a local Noetherian scheme whose closed point $s\in S$ has perfect residue field $k$. Let $\pi:X\to S$ be a projective morphism with $X$ regular, let $E\subset X$ be a simple normal crossings divisor, and let $\cL$ be a $\pi$-ample line bundle on~$X$. Then there exists $l\geq 1 $ and $\sigma\in H^0(X,\cL^{\otimes l})$ such that the zero-locus $D\subset X$ of $\sigma$ is regular, contains no irreducible component of $E$ and such that $D\cup E$ is a simple normal crossings divisor in $X$.
\end{prop}

\begin{proof}
If $l\gg 0$, and we choose such a $l$, then $\cL^{\otimes l}|_{X_s}$ is very ample and the restriction map $H^0(X,\cL^{\otimes l})\to H^0(X_s,\cL^{\otimes l}|_{X_s})$ is surjective by Serre vanishing. Let $(E_i)_{i\in I}$ be the irreducible components of $E$, and define $E_H:=\cap_{i\in H}E_i$ for $H\subset I$.
By Noetherian induction, we may write the $k$-variety 
$(E_{H,s})^{\red}$ as a disjoint union of finitely many smooth connected locally closed subvarieties ${(Y_{H,j}\subset X_s^{\red})_{j\in J(H)}}$, where $J(H)$ is a finite set of indices. By Bertini's theorem \cite[Th\'eor\`eme 6.10 2)]{Jouanolou} (if $k$ is finite, we rather use Poonen's \cite[Theorem 1.3]{Poonen} after maybe replacing $l$ with an appropriate multiple) applied to all the subvarieties $Y_{H,j}$ of $X_s^{\red}$ for varying $H\subset I$ and $j\in J(H)$, there exists $\tau\in H^0(X_s,\cL^{\otimes l}|_{X_s})$ such that the zero-locus of $\tau$ in $Y_{H,j}$ is smooth of codimension $1$ in $Y_{H,j}$. Let $\sigma\in H^0(X,\cL^{\otimes l})$ be such that $\sigma|_{X_s}=\tau$. Let $D\subset X$ be the zero-locus of~$\sigma$ and set $D_H:=D\cap E_H$ for $H\subset I$.

Fix $H\subset I$, and let $\Xi_H\subset D_H$ be the set of $x\in D_H$ such that $D_H$ is regular of codimension~$1$ in~$E_H$ at $x$. 
Choose $x\in D_{H,s}$, and let $j\in J(H)$ be such that $x\in Y_{H,j}$.
The inclusion $T_{D\cap Y_{H,j},x}\subset T_{Y_{H,j},x}$ is not an equality by our choice of $\tau$. It follows that the inclusion $T_{D_H,x}\subset T_{E_H,x}$ is not an equality either. Since $E_H$ is regular at $x$ and $D_H$ is defined, locally at $x\in E_H$, by the vanishing of a single equation, we deduce that $x\in \Xi$. We have shown that $D_{H,s}\subset \Xi$.
As $\Xi$ is stable by generization and $\pi|_{D_H}:D_H\to S$ is proper, we deduce that $\Xi=D_H$. This completes the proof of the proposition.
\end{proof}

\section{Sums of squares}
\label{sos}

This section is devoted to the proof of Theorem \ref{thmain}.

\subsection{Sums of squares and Galois cohomology}

If $X$ is a scheme on which $2$ is invertible, and if $a\in\cO(X)^*$, we denote by $\{a\}\in  H^1_{\et}(X,\Z/2)$ the image of $a$ by the boundary map of the Kummer exact sequence $0\to\Z/2\to\bG_m\xrightarrow{2}\bG_m\to 0$.

\begin{prop}
\label{carrecoho}
Let $F$ be a field of characteristic $\neq 2$, let $a\in F^*$, and choose $r\geq 0$. The following assertions are equivalent.
\begin{enumerate}[(i)]
\item One has $\{-1\}^{r}\cupp\{a\}=0\in H^{r+1}(F,\Z/2)$.
\item The element $a\in F^*$ is a sum of $2^{r}$ squares in $F$.
\end{enumerate}
\end{prop}

\begin{proof}
By the Milnor conjecture proven by Voevodsky \cite[Corollary 7.4]{Voevodsky}, statement~(i) is equivalent to the vanishing of the symbol $\{-1,\dots,-1,a\}\in K^M_{r+1}(F)/2$ in Milnor K-theory. By \cite[Corollary 3.3]{EL}, it is in turn equivalent to the Pfister form $\langle 1,1\rangle^{\otimes r}\otimes \langle 1,-a\rangle$ being isotropic. Since a Pfister form is isotropic if and only if it is hyperbolic \cite[Theoreme 1 und 2]{Pfistermult}, this is also equivalent to the isotropy of $\langle 1\rangle^{\oplus 2^{r}}\oplus\langle -a\rangle$, hence to condition (ii) by \cite[I, Corollary 3.5]{Lam}.
\end{proof}

\subsection{Level}
\label{level}

In \S\ref{level}, we study the level of function fields over Henselian local fields.

\begin{prop}
\label{proplevel}
Let $S$ be an integral Henselian excellent local scheme of dimension $\geq 1$ with closed point $s\in S$ whose residue field $k$ has characteristic $0$. Let $\pi:X\to S$ be a proper surjective morphism with $X$ regular, integral of dimension~$d$, and let $F$ be the function field of $X$.
\begin{enumerate}[(i)]
\item If $(X_s)_r\neq\varnothing$, then $s(F)=+\infty$.
\item  If $(X_s)_r=\varnothing$ and $\cd_2(k[\sqrt{-1}])\leq \delta$, then $s(F)\leq 2^{d+\delta-1}$. 
\end{enumerate}
\end{prop}

\begin{proof}
If $(X_s)_r\neq\varnothing$, then $\Spec(F)_r\neq\varnothing$ by Lemma \ref{gener} below, proving assertion~(i).

To prove (ii), we may assume that $\pi$ is projective and that $E:=X_s^{\red}$ is a simple normal crossings divisor in $X$, by Chow's lemma \cite[Th\'eor\`eme 5.6.1]{EGA2} and resolution of singularities \cite{Hironaka, Temkin}.
By Proposition \ref{Bertini}, there exists a regular divisor $D\subset X$ containing no irreducible component of $E$, such that $D\cup E$ is a simple normal crossings divisor in $X$ and such that $X\setminus D$ is affine.

Since the $k$-variety $U:=(X_s\setminus D_s)^{\red}$ is affine of dimension $d-1$, one has $\cd_2(U_{k[\sqrt{-1}]})\leq d+\delta-1$ by \cite[Expos\'e XIV, Corollaire 3.2]{SGA43} and by the hypo\-thesis that 
$\cd_2(k[\sqrt{-1}])\leq \delta$.
Since moreover ${U_r=\varnothing}$, Scheiderer \cite[Corollary~7.21]{Scheiderer} has shown that $\cd_2(U)\leq d+\delta-1$, hence that $H^{d+\delta}_{\et}(U,\Z/2)=0$.
Proposition \ref{purity} (iii)  yields an isomorphism $H^{d+\delta}_{\et}(X\setminus D,\Z/2)\simeq H^{d+\delta}_{\et}(U,\Z/2)=0$.

One has $\{-1\}^{d+\delta}=0\in H^{d+\delta}_{\et}(X\setminus D,\Z/2)$ since the whole group vanishes. As a consequence, $\{-1\}^{d+\delta}=0\in H^{d+\delta}(F,\Z/2)$.  Applying Proposition \ref{carrecoho} with $a=-1$ yields $s(F)\leq 2^{d+\delta-1}$, proving (ii).
\end{proof}

\begin{lem}
\label{gener}
Let $X$ be an integral regular scheme with function field $F$. Then any point of $X_r$ is in the closure of some point of $\Spec(F)_r\subset X_r$.
\end{lem}

\begin{proof}
Let $(x,\prec)\in X_r$, where $x\in X$ and $\prec$ is a field ordering of $\kappa(x)$.
Since $\kappa(x)$ is formally real, it has characteristic $0$.  As $\cO_{X,x}$ is regular, we can find a sequence 
$\cO_{X,x}=A_N\to\dots\to A_0=\kappa(x)$ of surjections of regular local rings such that the localization of $A_i$ at the kernel of $A_i\to A_{i-1}$ is a discrete valuation ring. Applying \cite[II \S 4, Th\'eor\`eme 2]{Corpslocaux} to the completions of these discrete valuation rings yields an inclusion $F\subset\kappa(x)((t_1))\dots((t_N)).$
By \cite[VIII, Proposition 4.11 (1)]{Lam}, the ordering $\prec$ of $\kappa(x)$ may be extended to an ordering $\prec'$ of $\kappa(x)((t_1))\dots((t_N))$. The description of $\prec'$ given in \emph{loc.\ cit.} shows that if the constant coefficient of $f\in \kappa(x)[[t_1,\dots,t_N]]$ is $\succ 0$, then $f\succ' 0$. Let $\prec_F$ be the restriction of $\prec'$ to $F$. The definition \cite[(0.4)]{Scheiderer} of the topology of $X_r$ shows that $(x,\prec)$ belongs to the closure of $(\Spec(F),\prec_F)$ in $X_r$, proving the lemma.
\end{proof}

The first assertion of Theorem \ref{thmain} follows easily from Proposition \ref{proplevel}.

\begin{proof}[Proof of Theorem \ref{thmain} (i)]
We may assume that $F$ is finitely generated over $K$. Define $S:=\Spec(A)$, and let $\pi:X\to S$ be a projective morphism with $X$ integral such that $F$ is the function field of $X$. Resolving singularities \cite{Hironaka,Temkin}, we may assume that $X$ is regular. It has dimension $d:=n+m$. Since $F$ is not formally real, Proposition \ref{proplevel}
shows that $s(F)\leq 2^{d+\delta-1}$. 
As $p(F)\leq s(F)+1$ for any field $F$ that is not formally real \cite[XI, Theorem 5.6~(2)]{Lam}, we deduce that $p(F)\leq 2^{d+\delta-1}+1$.
\end{proof}

\subsection{Pythagoras number}
\label{pythagoras}

We now deduce the two last assertions of Theorem \ref{thmain} from the first.

\begin{proof}[Proof of Theorem \ref{thmain} (ii)]
Let $a\in F^*$ be a sum of squares. Since $F$ is formally real, $-a$ is not a square in $F$. We consider the field extension $L:=F[\sqrt{-a}]$ of $F$. One has $s(L)\leq 2^{n+m+\delta-1}$ by Theorem \ref{thmain} (i) because $L$ is not formally real. That $a$ is a sum of $2^{n+m+\delta}-1$ squares in $F$ follows from \cite[Chapter 11, Theorem~2.7]{vieuxLam}.
\end{proof}

\begin{proof}[Proof of Theorem \ref{thmain} (iii)]
Let $a\in K^*$ be a sum of squares, and consider the class $\alpha:=\{-1\}^{n+\delta-1}\cupp\{a\}\in H^{n+\delta}(K,\Z/2)$.
If $D\subset S:=\Spec(A)$ is an integral divisor with generic point $\eta_D$, we let $\res_D(\alpha)\in H^{n+\delta-1}(\kappa(\eta_D),\Z/2)$ be the residue of~$\alpha$ along $D$ \cite[\S 3.3]{CTBarbara}. 
It follows from \cite[Proposition 1.3]{CTOj} that $\res_D(\alpha)=e\{-1\}^{n+\delta-1}$, where $e\in\Z$ is the order of vanishing of $a$ along $D$.

Completing $A$ at $\eta_D$ yields an embedding $K\subset \kappa(\eta_D)((t))$. Since $a$ is a sum of squares in $K$ hence also in $\kappa(\eta_D)((t))$, either the $t$-adic valuation of $a\in \kappa(\eta_D)((t))$ is even, or $\kappa(\eta_D)$ is not formally real, by \cite[Proposition~4.2]{BGvG}. In the first case, $e$ is even and $\res_D(\alpha)=0$. In the second case, one has $n\geq 2$ since $k$ is formally real. It is thus possible to apply Theorem \ref{thmain}~(i) to the coordinate ring $\cO(D)$ of~$D$. This shows that $s(\kappa(\eta_D))\leq 2^{n+\delta-2}$, hence that $\{-1\}^{n+\delta-1}=0\in H^{n+\delta-1}(\kappa(\eta_D),\Z/2)$ by Proposition \ref{carrecoho}. Consequently, $\res_D(\alpha)=0$.

We have shown that the residues of $\alpha$ along all integral divisors $D\subset S$ vanish. Since $A$ is regular, applying the Gersten conjecture proven in this context by Panin \cite[Theorem C]{Panin} shows that $\alpha$ lifts to a class $\beta\in H^{n+\delta}_{\et}(S,\Z/2)$. 
Let $R$ be a real closed extension of $K$. Since $a$ is a sum of squares in $K$, it is a square in $R$, and it follows that $\beta|_R=\alpha|_R=0\in H^{n+\delta}(R,\Z/2)$.
By Lemma \ref{injec} below, one has $\beta=0$, hence $\alpha=0$, and Proposition \ref{carrecoho} implies that $a$ is a sum of $2^{n+\delta-1}$ squares~in~$K$.
\end{proof}

We have used the following lemma.

\begin{lem}
\label{injec}
Let $S$ be the spectrum of an integral Henselian regular local ring with residue field $k$ and fraction field $K$, and let $\beta\in H^q_{\et}(S, \Z/2)$. If $q>\cd_2(k[\sqrt{-1}])$ and if $\beta|_R=0\in H^q(R,\Z/2)$ for all real closed extensions $R$ of $K$, then $\beta=0$.
\end{lem}

\begin{proof}
The case where $k$ has characteristic $2$ is trivial since $k=k[\sqrt{-1}]$ and the restriction map $H^q_{\et}(S, \Z/2)\to H^q(k, \Z/2)$ is an isomorphism by proper base change \cite[Expos\'e XII, Corollaire 5.5 (iii)]{SGA43}. Assume now that the characteristic of $k$ is~$\neq 2$.

We set $k_r:=\Spec(k)_r$ and $G:=\Z/2$, and we consider the commutative diagram
\begin{equation}
\label{cdb}
\begin{aligned}
\xymatrix
@R=0.3cm 
@C=0.45cm 
{
H^q_{\et}(S, \Z/2)\ar^{\wr}[d]\ar[r]&H^q_G(S_r,\Z/2)\ar[d]\ar@{=}[r]& \bigoplus_{i=0}^q H^i(S_r,\Z/2)\ar[d]\ar[r]&H^0(S_r,\Z/2)\ar[d]\\
H^q(k, \Z/2)\ar^(.47){\sim}[r]&H^q_G(k_r,\Z/2)\ar@{=}[r]& \bigoplus_{i=0}^q H^i(k_r,\Z/2)\ar^(.57){\sim}[r]&H^0(k_r,\Z/2)
}
\end{aligned}
\end{equation}
whose vertical maps are restriction maps, whose right horizontal arrows are the projections, and whose other arrows are the one appearing in \cite[(7.19.1)]{Scheiderer}.  More precisely, the left horizontal arrows of (\ref{cdb}) are the maps \cite[(6.6.3)]{Scheiderer} applied with $A=\Z/2$, taking into account \cite[Corollary 6.6.1]{Scheiderer} and using the 
fact that the topoi associated to $X_r$ and  to the real \'etale site $X_{\ret}$ of $X$ are naturally equivalent \cite[Theorem~1.3]{Scheiderer}, and the middle horizontal equalities of (\ref{cdb}) are obtained by taking $C=X_r$ and $k=\Z/2$ in \cite[Corollary 6.3.2]{Scheiderer}.

As explained in \cite[(7.19.1)]{Scheiderer}, if $\xi\in S_r$ corresponds to a point $x\in S$ and to an ordering $\prec$ of $\kappa(x)$, and if $R$ is the associated real closure of $\kappa(x)$, then the image of $\beta$ by the first line of (\ref{cdb}) has value $0$ at $\xi$ if and only if $\beta|_R=0\in H^q(R,\Z/2)$. This is the case for all $\xi\in \Spec(K)_r\subset S_r$ by hypothesis. Since $\Spec(K)_r$ is dense in $S_r$ by Lemma \ref{gener} and by regularity of $S$, we deduce that $\beta$ vanishes in the upper right corner of (\ref{cdb}), hence in the lower right corner of (\ref{cdb}).

On the other hand, the left vertical arrow of (\ref{cdb}) is an isomorphism by proper base change \cite[Expos\'e XII, Corollaire 5.5 (iii)]{SGA43}, and the lower left horizontal arrow of (\ref{cdb}) is an isomorphism by \cite[Corollary 7.10]{Scheiderer} applied with $A=\Z/2$ and by the hypothesis that $q>\cd_2(k[\sqrt{-1}])$.
Moreover, since $k_r$ is Hausdorff, compact and totally disconnected \cite[VIII, Theorem~6.3]{Lam}, the global sections functor for abelian sheaves on $k_r$ is exact, showing that $H^i(k_r,\Z/2)=0$ for $i>0$, hence that the lower right horizontal arrow of (\ref{cdb}) is also an isomorphism. The commutativity of (\ref{cdb}) now shows that $\beta=0$.
\end{proof}

\begin{rem}
The bottom line of diagram (\ref{cdb}) goes back to the work of Arason, Elman and Jacob \cite{AEJ} (see especially  Theorem 2.3, Proposition 2.4 and the proof of Corollary 2.8 in \emph{loc.\ cit.}). Scheiderer's book \cite{Scheiderer} contains far-reaching generalizations of these results.
\end{rem}

\subsection{Optimality}
\label{optimality}

We now show the optimality of Theorem \ref{consmain}, conditionally upon Pfister's inequalities $p(\R(x_1,\dots,x_n))\leq 2^n$ being equalities.

\begin{prop}
\label{propopt}
Assume that $p(\R(x_1,\dots,x_{n-1}))=2^{n-1}$ for some $n\geq 1$. Then: 
\begin{enumerate}[(i)]
\item One has $p(\R((x_1,\dots,x_{n})))=2^{n-1}$.
\item There exists a finite extension $F$ of $\R((x_1,\dots,x_{n}))$ such that $s(F)=2^{n-1}$ and $p(F)=2^{n-1}+1$.
\end{enumerate}
\end{prop}

\begin{proof}
(i) This was proven by Hu in \cite[Corollary 2.3]{Hupyth}.

(ii) Let $f\in\R(x_1,\dots,x_{n-1})$ be a sum of squares that is not a sum of $2^{n-1}-1$ squares in $\R(x_1,\dots,x_{n-1})$. The field $L:=\R(x_1,\dots,x_{n-1})[\sqrt{-f}]$ is such that $s(L)\geq 2^{n-1}$ by \cite[Chapter 11, Theorem~2.7]{vieuxLam}. Let $Z$ be a smooth projective integral variety over $\R$ with $\R(Z)=L$. Since $L$ is not formally real, one has $Z(\R)=\varnothing$ by Lemma \ref{gener}. Embed $Z$ in a real projective space, and consider the cone $C$ over $Z$ in this embedding with vertex $p\in C$. Define $A:=\widehat{\cO_{C,p}}$ and $F:=\Frac(A)$. 
By \cite[Theorems 15 and 16]{Cohen} (see also the footnote (19) in \emph{loc.\ cit.}), there exists an injection $\R[[x_1,\dots,x_n]]\subset A$ endowing $A$ with a structure of finite $\R[[x_1,\dots,x_n]]$-algebra; it follows that $F$ is a a finite extension of $\R((x_1,\dots,x_{n}))$.

  Let $\pi:X\to \Spec(A)$ be the blow-up of the closed point. The scheme $X$ is regular and the exceptional divisor of $\pi$ is isomorphic to $Z$. By Proposition \ref{proplevel}, $F$ is not formally real. As $L$ is the residue field of a valuation on $F$, \cite[Proposition~4.3]{BGvG} shows that $p(F)\geq s(L)+1\geq 2^{n-1}+1$. 
By \cite[XI, Theorem 5.6~(2)]{Lam}, one has $s(F)\geq p(F)-1\geq 2^{n-1}$. That these inequalities are in fact equalities follows from Theorem \ref{consmain}~(ii) and \cite[XI, Theorem 5.6~(2)]{Lam}.
\end{proof}

\bibliographystyle{plain}
\bibliography{localPythagoras}

\end{document}